\documentclass[a4paper,12pt]{article}
\usepackage{graphicx,amssymb,amstext,amsmath}

\usepackage[latin1]{inputenc}
\usepackage{amsthm}
\usepackage{dsfont}
\usepackage{amsmath}
\usepackage{amsfonts}
\usepackage{amssymb}
\usepackage{mathrsfs}
\usepackage{graphicx}
\usepackage{cancel}
\usepackage{calligra}
\usepackage{graphicx,color}
\usepackage[all]{xy}
\usepackage[total={5.2in,9.1in},
top=1.4in, left=1.55in, includefoot]{geometry}

\usepackage{amsmath}
\usepackage{amssymb}
\usepackage{amsthm}
\newtheorem{teo}{Theorem}[section]
\newtheorem{lem}[teo]{Lemma}

\newtheorem{coro}[teo]{Corollary}
\newtheorem{defi}[teo]{Definition}

\newtheorem{remark}[teo]{Remark}

\newtheorem{conj}{Conjecture}

\begin{document}
\thispagestyle{empty}

\title{\textbf{\Large{Rigidity of Lyapunov Exponents for Geodesic Flows }}}
\author{Nestor Nina Zarate\footnote{Partially Supported By Capes and CNPq} \and Sergio Roma\~na\footnote{Suppoted by Bolsa Jovem Cientista do Nosso Estado No. E-26/201.432/2022 - Brazil, NNSFC 12071202, and NNSFC 12161141002 from China.} }
\date{}
\maketitle

\textit{\,\,\,\,\,\,\,\,\,\,\,\,\,\,\,\,\,\,\,\,\,\,\,\,\,\,\,\,\,\,\,\,\,\,\,\,\,\,\,\,\,\,\,\,  This paper is dedicated to the memory of Nestor Nina Zarate }
 
\begin{abstract}
In this paper, we study rigidity problems between Lyapunov exponents along periodic orbits and geometric structures.  More specifically, we prove that for a surface  $M$ without focal points, if the value of the Lyapunov exponents is constant over all periodic orbits, then  $M$ is the flat $2$-torus or a surface of constant negative curvature.  We obtain the same result for the case of Anosov geodesic flow for surface, which generalizes  C. Butler's result  \cite{Bu} in dimension two. Using completely different techniques, we also prove an extension of \cite{Bu} to the finite volume case, where the value of the Lyapunov exponents along all periodic orbits is constant, being the maximum or minimum possible.

\begin{quote}
\textbf{Keywords: Lyapunov Exponent, Rigidity, Anosov flow.}
\end{quote}

\end{abstract}

\section{Introduction.}
Rigidity problems in dynamics and geometry are very interesting research topics in mathematics because they provide us with a very deep connection between two different areas. Many researchers around the world are unraveling different types of rigidity between these two areas. We can cite several of these authors (\cite{BCG, Con, G, CFF, Co, CK, CK2, FO, FK, F, Ot1, Ot2}. 
We invite the reader to check the introduction of the articles \cite{IR2} and \cite{Bu}, which provide an overview of the different results that have already been obtained. \\
In this paper, we focus our attention on rigidity problems of Lyapunov exponents. We will see that, under certain conditions on the Lyapunov exponents of an Anosov geodesic flow (of geodesic flow of surface without focal points), we can obtain rigidity on the sectional curvature. 
Our goal is to extend the result over the rigidity of equality of Lyapunov exponents for geodesic flows of Butler in \cite{Bu} for dimension two. \\
\indent Let $\theta$ be a periodic orbit of period $\tau$ for the geodesic flow $\phi^{t}\colon SM\to SM$ for a Riemannian manifold $M$.  Let $\chi_{1}{(\theta)},\ldots,\chi_{m-1}{(\theta)}$ be the complex eigenvalues of $D_{\theta}\phi^{\tau}:E_{\theta}^{u}\to E_{\theta}^{u}$, counted with the multiplicity. A beautiful result due to C. Butter \cite{Bu} stated that:

\begin{teo}\emph{\cite{Bu}}\label{teo4.1}\, 
Let $M$ be an $m$-dimensional compact negatively curved Riemannian manifold. Suppose that  
\begin{eqnarray*}
\left|\chi_{i}{(\theta)}\right|=\left|\chi_{j}{(\theta)}\right|; \hspace{0,2cm} 1\leq i,j\leq m-1,
\end{eqnarray*}
for every periodic point $\theta$ of the geodesic flow $\phi:SM\to SM$. Then $M$ is homothetic to a compact quotient of $\mathbb{H}_{\mathbb{R}}^{m}.$
\end{teo}   

In particular, Butler's result above claims that if each periodic orbit of the geodesic flow defined on a manifold of negative curvature, has exactly one Lyapunov exponent on the unstable (or stable) bundle then the manifold has constant negative curvature. 

\noindent Since the geodesic flow on a compact manifold of negative curvature is an Anosov flow, then supported by Butler's result, in \cite{IR} the authors proposed the following conjecture (see \cite[Conjecture 1]{IR}).

\begin{conj}\label{conj.1}Let $M$ be a complete Riemannian manifold with finite volume, whose
geodesic flow is Anosov. If the unstable Lyapunov exponents are constant in all periodic orbits, then $M$ has constant negative sectional curvature.
\end{conj}

\noindent The main goal of this paper is to prove the last conjecture in dimension two. In fact, we prove some rigidity results for conservative Anosov flows on manifolds of dimension three, which encompasses the case of geodesic Anosov flows of compact surfaces.\\ 

We said that a flow $\phi^t$ is conservative if it preserves a smooth volume measure. In this context, our first result is:

\begin{teo}\label{New Theorem}
Let $\phi^t: N\to N$ be a $C^k$-conservative Anosov flow with $k\geq 5$ on a three-dimensional manifold $N$. The following are equivalents: 
\begin{itemize}
\item[\emph{(1)}] There is $\alpha>0$ such that for all $\theta\in Per(\phi)$ and $\xi\in E_{\theta}^{u}\backslash\{0\}$ $$\chi^{+}(\theta,\xi)=\lim_{t\to +\infty}\dfrac{1}{t}\log\|d_{\theta}\phi^{t}(\xi)\|=\alpha.$$

\item[\emph{(2)}] The volume $m$ is a measure of maximal entropy.

\item[\emph{(3)}]For any $\epsilon>0$, $\phi^t$ is $C^{k-\epsilon}$  conjugate to an Algebraic flow.
\end{itemize}
\end{teo} 

Geodesic flows on compact manifolds are natural examples of conservative flows, in fact, the \text{Liouville} measure is always an invariant measure for the geodesic flow. Therefore, using the structure of the geodesic flow, we use the Theorem \ref{New Theorem}
to prove the Conjecture \ref{conj.1} for surfaces. More specifically 

\begin{teo}\label{teo4.6}\emph{[\textbf{Main Theorem}]}\, Let $\phi^{t}:SM\to SM$ be the $C^{k}$- Anosov geodesic flow with $k\geq 5$ on a compact surface $M$. The following are equivalents:
\begin{itemize}
\item[\emph{(1)}] There is $\alpha>0$ such that for all $\theta\in Per(\phi)$ and $\xi\in E_{\theta}^{u}\backslash\{0\}$ $$\chi^{+}(\theta,\xi)=\lim_{t\to +\infty}\dfrac{1}{t}\log\|d_{\theta}\phi^{t}(\xi)\|=\alpha.$$

\item[\emph{(2)}] The surface $M$ has constant negative curvature $K_{M}=-\alpha^{2}$.

\end{itemize}
\end{teo}
\noindent The last two results are also valid looking for the Lyapunov exponents on the stable subbundle.
\ \\

The theorem \ref{teo4.6} extends Butler's result in dimension two since we just assume the Anosov condition of geodesic flow without any restrictions on the surface curvature.
The main idea behind the proof of this theorem is to use Kalinin's result (cf. \cite{Kal}) to show that the Liouville measure is a maximum entropy measure (MME), and then find some geometric rigidity for this fact. For this sake, we prove that the hypotheses of \ref{teo4.6} are equivalent to having the Liouville measure as an MME in any dimension. More precisely,  

\begin{teo}\label{teo4.14}
Let $M$ be a compact Riemannian manifold with Anosov geodesic such that the unstable \emph{(}or stable\emph{)} Lyapunov exponents are constant along the periodic orbits, then the Liouville measure $\mathcal{L}$ is a measure of maximal entropy, \emph{i.e.}, $h_{top}(\phi)=h_{\mathcal{L}}(\phi).$
\end{teo}

\noindent The Theorem \ref{teo4.14} is linked to the \emph{Katok Entropy Conjecture} (see Conjecture \ref{conj411}). \\
Furthermore, due to a good understanding of the geodesic flow of surfaces, we can use Theorem \ref{teo4.14} to obtain a very nice rigidity for the geodesic flows of surfaces without focal points.



\begin{teo}\label{Teo_New_1}Let $M$ be a compact surface without focal points with a $C^2$ Riemannian metric and let $\phi^t$ be its geodesic flow. The following are equivalents:
\begin{itemize}
\item[\emph{(1)}] There is $\alpha\geq 0$ such that for all $\theta\in Per(\phi)$ and $\xi\in G_{\theta}^{u}\backslash\{0\}$ $$\chi^{+}(\theta,\xi)=\lim_{t\to +\infty}\dfrac{1}{t}\log\|d_{\theta}\phi^{t}(\xi)\|=\alpha,$$
where $G^{u}_{\theta}$ is the unstable Green subdundle. 
\item[\emph{(2)}] The surface $M$ has curvature $K_M=-\alpha^2$. Moreover, if $\alpha=0$, then $M=\mathbb{T}^2$, the flat torus.
\end{itemize}
\end{teo}
\noindent Note that, in the last theorem, we can relax the high differentiability condition in the Theorem \ref{teo4.6} and even relax the hyperbolicity condition. Theorem \ref{teo4.6} and Theorem \ref{Teo_New_1} may seem similar, but they are not. First, note that the Anosov condition does not imply the absence of focal points. Furthermore, they have different differentiability classes. \\

Finally, using techniques completely different from those used in the proof of all previous theorems, we extend Butler's results in the case of finite volume and pinched negative curvature,  when the value of the Lyapunov exponents is maximum or minimum,  more specifically:

\begin{teo}\label{thm4.5}
Let $M$ be a complete Riemannian manifold of finite volume and such that $-c^{2}\leq K_{M}\leq -b^{2}<0$. Let $\phi^{t}\colon SM\to SM$ be the geodesic flow. Consider $\alpha\in \{b, c\}$ and assume that 
for all $\theta\in Per(\phi^t)$ we have
 $$\chi^{+}(\theta,\xi)=\lim_{t\to+\infty}\dfrac{1}{t}\log\|d_{\theta}\phi^{t}(\xi)\|=\alpha,$$
for all $\xi\in E_{\theta}^{u}\backslash\{0\}$. Then $K_{M}=-\alpha^{2}$.
\end{teo}


\noindent The Theorem \ref{thm4.5} gives us the proof of a weak version of conjecture \ref{conj.1} in the case of pinched negative curvature when the Lyapunov exponents are the minimum or maximum possible. The conjecture remains open for the general case, even on compact manifolds. \\
\ \\
\noindent \textbf{Structure of the Paper:}  In section 2 we provide the tools to prove the main results. In section 3, we proved Theorem \ref{teo4.14} and later we will use it in the proof of Theorem \ref{New Theorem} and Theorem \ref{teo4.6}. We finish section 3 with the proof of Theorem \ref{Teo_New_1} using surface techniques without focal points. Finally, section 4 will be dedicated to the proof of Theorem \ref{thm4.5}.

\section{Preliminaries and Notation}
Throughout this paper, $M$ will denote a complete Riemannian manifold without
boundary of dimension $n\geq 2$, $TM$ is the tangent bundle, $SM$ its unit tangent bundle, and $\mathcal{L}$ will be its Liouville measure.

\subsection{Geodesic Flow}
For each $t\in \mathbb{R}$ and $\theta=(p,v)\in SM$ consider the family of diffeomorphism  $\phi^t\colon SM \to SM$ defined by 
$$\phi^t(\theta)=(\gamma_{\theta}(t),\gamma_{\theta}'(t)),$$
where $\gamma_{\theta}(t)$ is the unique geodesic with initial conditions $\gamma_{\theta}(0)=p$ and $\gamma_{\theta}'(0)=v$
This family is called the  \emph{geodesic flow} defined over $SM$.\\ 
The spray vector field $G$ is the vector field derivative of $\phi^{t}$, that is, 
$G(\theta):=\frac{d}{dt}\phi^{t}(\theta)\Big|_{t=0}.$

Given $\theta=(x,v)\in SM$, we identify $T_{\theta}TM $ with $T_xM\oplus T_xM$ using the identification 
$\xi \to (d\pi(\xi), \mathcal{K}_{\theta}(\xi))$, where $\pi\colon TM\to M$ denotes the canonical projection and $\mathcal{K}\colon TTM \to TM$ the connection map defined via the Levi-Civita connection.
The last identification allows us to define \emph{Sasaki's metric} on $TM$ as 
$$ \langle \xi,\eta\rangle_{\theta} = \langle D_{\theta}\pi(\xi), D_{\theta}\pi(\eta)  \rangle  +  \langle \mathcal{K}_{\theta}(\xi),    \mathcal{K}_{\theta}(\eta)   \rangle. $$

\noindent The geodesic flow  $\phi^t$ is \emph{Anosov} if  the tangent bundle of $SM$, $T(SM)$,  has a splitting $T(SM) = E^s \oplus \langle G \rangle \oplus E^u $ such that 
\begin{eqnarray*}
	D(\phi^t_{_{M}})_{\theta} (E^s(\theta)) &=& E^s(\phi_{_{M}}^t(\theta)),\\
	D(\phi^t_{_{M}})_{\theta} (E^u(\theta)) &=& E^u(\phi_{_{M}}^t(\theta)),\\
	||D(\phi^t_{_{M}})_{\theta}\big{|}_{E^s}|| &\leq& C \lambda^{t},\\
	||D(\phi^{-t}_{_{M}})_{\theta}\big{|}_{E^u}|| &\leq& C \lambda^{t},
\end{eqnarray*}
for all $t\geq 0$ with $ C > 0$ and $0 < \lambda <1$, where $G$ is the geodesic vector field. In that definition, we are always using \emph{Sasaki's metric} of $SM$ (see \cite{Pa} for more details). \\
Natural examples of Anosov geodesic flows are produced by metrics of negative curvature on compact manifolds (cf. \cite{Anosov}) and metrics of pinched negative curvature\footnote{That means there are $b,c>0$ such that the sectional curvature satisfies $-c^2\leq K\leq -b^2$.} in the non-compact case (\cite{Kn}). However, other examples can be found without the last assumption on the curvature (see by example \cite{Ebe} and \cite{IR1}).

\subsection{No conjugate points and No focal Points}\label{sec-ncp}
Given  $\gamma$ we said that a vector field $J$ along a geodesic $\gamma$ is a Jacobi field if it satisfies the Jacobi equation 
$$J''(t)+R(\gamma'(t),J(t))\gamma'(t)=0,$$
where $R$ is the curvature tensor of $M$ and ``$'$" denotes the covariant derivative along $\gamma$. Jacobi's field plays an important role in the study of the dynamic of the geodesic flow, in fact, the Jacobi field can be used to get geometric properties when we have some dynamic properties of the geodesic flow (see \cite{Pa} and the introduction of \cite{IR} for more details).\\
Another important observation is based on the following fact: If $\xi\in T_{\theta}SM$, then $J_{\xi}(t)$ denotes the unique Jacobi field along $\gamma_{\theta}(t)$ such that $J_{\xi}(0)=D\pi_{\theta}(\xi)$ and $J'(\xi)(0)=\mathcal{K}_{\theta}(\xi)$. Moreover,

\begin{equation}\label{EQ-Dif_Jacobi}
D\phi^t_{\theta}(\xi)=(J_{\xi}(t), J'_{\xi}(t)).
\end{equation}
The last formula allows us to focus our attention on the study of the Jacobi field to understand the behavior of the geodesic flow. \\
\indent Two points $p,q\in M$ are $\textit{conjugate}$ if there is a geodesic $\gamma$ joining $p$ and $q$ and a non-zero Jacobi field along $\gamma$ that vanishes at $p$ and $q$. When neither two points in $M$ are conjugate, we say the manifold $M$ $\textit{has no conjugate points}$. Another important kind of manifold is the manifold without focal points, we say that a manifold $M$ $\textit{has no focal points}$, if for any unit speed geodesic $\gamma$ in $M$ and for any Jacobi field $J$ on $\gamma$ such that $J(0)=0$ and $J'(0)\neq 0$ we have $(\|J\|^{2})'(t)>0$, for any $t>0$. It is easy to see that if a manifold has no focal points, then it has no conjugate points.  

The more classical examples of manifolds without focal points and therefore without conjugate points are the manifolds of non-positive sectional curvature. It is possible to construct a manifold having positive curvature somewhere, and without conjugate points. There is some special connection between the Anosov property and no conjugate points property. In fact, in \cite{Klin} Klingenberg proved that compact manifold with Anosov geodesic flow has no conjugate points, in result was generalized Ma\~n\'e in the case of finite volume (cf. \cite{Ma}). Recently, the second author in a joining work with I. Melo proved the result for the case of infinite volume and under the assumption of sectional curvature bounded below (cf. \cite{IR2}).\\
In \cite{Ebe}, Eberlein obtains a general characterization of the Anosov condition for compact manifold without focal points, this characterization will be useful in the proof of Theorem \ref{Teo_New_1}.

\subsubsection{Green Subbundles}
For $\xi = (w_1, w_2) \in T_{\theta}SM$, where $w_1, w_2\in T_{p}M$ with $\langle w_i, v\rangle=0$, $i=1,2$,  we denote by $J_{\xi}(t)$ the unique Jacobi vector field along $\gamma_{\theta}$ such that $J_{\xi}(0) = w_1$ and $J_{\xi}'(0)=w_2$. For more details see \cite{Pa}.\\
To study the dynamic behavior of the geodesic flow, we usually look for two special subbundles of $TSM$, which are invariant for $D\phi^t$. The \emph{stable and unstable Green subbundles} are defined, respectively, as follows:
$$G^s_\theta= \{\xi \in T_{\theta}SM : \xi \,\, \text{is orthogonal to} \,\, G(\theta) \,\, \text{and}\,\, ||J_{\xi}(t)||\,\,\,\text{is bounded for}\,\,\ t \geq 0\}$$
\vspace{-0.1cm} 
and
\vspace{-0.4cm} 
$$G^u_\theta= \{\xi \in T_{\theta}SM : \xi \,\, \text{is orthogonal to} \,\, G(\theta)\,\  \text{and} ||J_{\xi}(t)||\,\,\, \text{is bounded for}\,\,\ t \leq 0\},$$
where $G(\theta)$ is the geodesic vector field.\\

\noindent If $M$ has dimension $n$ and has no conjugate points, then the
dimension of Green subbundles is $n-1$. Moreover, if $M$ has no focal points, then 
Green's subbundles depend continuously on $\theta\in SM$ (cf. \cite{Ebe}). In fact, for manifolds without focal points, the Anosov condition is characterized by the condition $G^s_{\theta}\cap G^u_{\theta}=\{0\}$ and clearly $G^*_{\theta}=E^{*}_{\theta}$, $*=s,u$.(see \cite{Ebe} for more details on Green subbundles). 
\subsubsection{Jacobi Tensor and Riccati Equation}

Let $\gamma_{\theta}$ be a geodesic and consider $V_{1}, V_{2},\ldots, V_{n}$ a system of parallel orthonormal vector fields along $\gamma_{\theta}$ with $V_{n}(t)=\gamma'_{\theta}(t)$. Any orthogonal vector field $Z(t)$ along $\gamma_{\theta}$ can be write as 
$$Z(t)=\sum_{i=1}^{n-1}y_{i}(t)V_{i}(t).$$

We can identified $Z(t)$ with the curve  $\alpha'(t)=(y'_{1}(t), y'_{2}(t),\dots , y'_{n-1}(t))$. Conversely, any curve in $\mathbb{R}^{n-1}$ can be identified with a perpendicular vector field on $\gamma_{\theta}(t)$.

For each $t\in \mathbb{R}$, consider the symmetric matrix $R(t)=(R_{i,j}(t))$ defined as  $$R_{i,j}(t)=\langle R(\gamma'_{\theta}(t),V_{i}(t))\gamma'_{\theta}(t), V_{j}(t))\rangle,$$ 
 $1\leq i,j\leq n-1$, where $R$ is the curvature tensor of $M$. Thus, we can consider the matrix \emph{Jacobi equation} associated to $R$ as 
\begin{eqnarray}\label{ec1.1}
Y''(t)+R(t)Y(t)=0.
\end{eqnarray}

If $Y(t)$ is a solution of (\ref{ec1.1}) then for each $x\in\mathbb{R}^{n-1}$, the curve $J(t)=Y(t)x$ corresponds to a Jacobi perpendicular vector field on $\gamma_{\theta}(t)$. In the following, we give a slight description of some special solution of the equation (\ref{ec1.1}). For $\theta\in SM$, $r\in \mathbb{R}$, we consider $Y_{\theta,r}(t)$ be the unique solution of (\ref{ec1.1}) satisfying $Y_{\theta,r}(0)=I$ and $Y_{\theta,r}(r)=0$. In the case of a manifold without conjugate points, in \cite{L}, Green proved that $\displaystyle\lim_{r\to -\infty} Y_{\theta,r}(t)$ exists for all $\theta\in SM$ (see also \cite{Ebe}, Sect. 2). Moreover, if we define: 
\begin{eqnarray*}\label{ec1.2}
Y_{\theta,u}(t):=\lim_{r\to -\infty} Y_{\theta,r}(t),
\end{eqnarray*}
we obtain a solution of Jacobi equation (\ref{ec1.1}) such that $\det Y_{\theta,u}(t)\neq 0$, which we call the \emph{unstable Jacobi Tensor}. Furthermore, in \cite{L} was prove that for all $t\in \mathbb{R}$
\begin{equation}\label{EQ-Unstable}
U_{\theta, u}(t):=U_{u}(\phi^{t}(\theta))=\dfrac{DY_{\theta,u}}{dt}(t)Y^{-1}_{\theta,u}(t)
\end{equation}
is well defined and it is a symmetric solution of the matrix Riccati equation:
\begin{eqnarray}\label{ec1.3}
U'(t)+U^{2}(t)+R(t)=0.
\end{eqnarray}
Analogously, taking the limit when $r\to +\infty$, we have defined $U_{\theta, s}(\theta)$, that also satisfies the Riccati equation (\ref{ec1.3}). When the curvature is bounded below by $-c^2$, then (see \cite{L} and \cite{Ebe})
\begin{equation}\label{EQ-Green}
\max\{\|U_{\theta, s}(t)\|, \|U_{\theta, u}(t)\|\}\leq c.
\end{equation}

When the sectional curvature is negative bounded above, that is, there is a positive constant $b$ such that $K\leq -b^2$, then (cf. \cite{Ebe} and \cite{Kn})
\begin{equation}\label{ec4.8}
\langle U_{\theta,u}(t)x,x\rangle\geq b\langle x,x\rangle
\end{equation}
and
\begin{equation*}
\langle U_{\theta,s}(t)x,x\rangle\leq -b\langle x,x\rangle,
\end{equation*}
for all $t\in \mathbb{R}$.
\subsection{Lyapunov Exponents}
Given $\theta\in SM$ and $\xi\in T_{\theta}SM$, the \emph{Lyapunov exponent} associated to $(\theta, \xi)$ is defined as 
$$\chi(\theta, \xi):=\lim_{t\to \infty}\frac{1}{t}\log ||D\phi^t_{\theta}(\xi)||,$$
whenever the limit exists.\\
Oseledet's Theorem guarantees the existence of Lyapunov exponent for geodesic flows on a compact manifold, that is,  there exists a set $\Lambda$ of full $\mathcal{L}$-measure, and a good filtration of subspaces of $T_{\theta}SM$ (see \cite{Osel}) where the Lyapunov exponents always exist.  
It is not difficult to proof that, if $M$ has no conjugate points, then for $\theta\in \Lambda$  $$\chi(\theta, \xi)\leq 0, \,\, \xi\in G^{s}_{\theta}$$
$$\chi(\theta, \eta)\geq 0, \,\, \eta\in G^{u}_{\theta}.$$
Thus, wherever we use $\chi^{+}(\theta, \eta)$, we look always for $\eta\in G^{u}_{\theta}$.
Moreover, we have the following characterization for manifold without conjugate points (\cite{MF} for compact case and \cite{IR2} for non-compact case):

\begin{equation}\label{EQ-Form-Lyp-Exp}
\lim_{t\to \infty}\frac{1}{t}\log |\text{det}|D\phi^t_{\theta}|_{G^u_{\theta}}|=\lim_{t\to \infty}\frac{1}{t}\int_{0}^{t}\text{tr}\,U_{\theta,u}(s)ds,
\end{equation}
where $U_{\theta,u}(s)$ is given by (\ref{EQ-Unstable}).\\
In the particular case of dimension two, the last formula has a very easy interpretation: Let $\eta\in G^{u}_{\theta}$ and $J_{\eta}(t)$ the Jacobi field associated to $\eta$ (as Section \ref{sec-ncp}). In dimension two, $J_{\eta}(t)=f_{\eta}(t)E(t)$, where $E(t)$ is a parallel orthonormal vector field along $\gamma_{\theta}(t)$ and $f_{n}(t)$ satisfies the uni-dimensional Jacobi equation 
$$J''(t)+K(\gamma_{\theta}(t))J(t)=0,$$
where $K(\gamma_{\theta}(t))$ is the sectional curvature along $\gamma_{\theta}(t)$.\\
Also $u_{\theta}(t)=\dfrac{f'_{\eta}(t)}{f_{\eta}(t)}$ satisfies the uni-dimensional Riccati equation
\begin{equation}\label{EQ-Uni-Riccati}
u'(t)+u^{2}(t)+K(\gamma_{\theta}(t))=0.
\end{equation}
From (\ref{EQ-Dif_Jacobi}), (\ref{EQ-Green}), and (\ref{EQ-Form-Lyp-Exp}) we obtain 
$$\chi^{+}(\theta, \eta)=\lim_{t\to \infty}\frac{1}{t}\log ||J_{\eta}(t)||=\lim_{t\to \infty}\frac{1}{t}\int_{0}^{t}u_{\theta}(s)ds.$$
Since $u_{\theta}(s)$ is bounded (see (\ref{EQ-Green})) and $u_{\theta}(s)$ satisfies (\ref{EQ-Uni-Riccati}), then using Cauchy-Schwarz inequality is easy to see that 
\begin{equation*}\label{Lyapunov and Curvature}
\chi^{+}(\theta, \eta)\leq \sqrt{-\limsup_{t\to \infty}\frac{1}{t}\int_{0}^{t}K(\gamma_{\theta}(t))ds}.
\end{equation*}
Finally, observe that if $\theta$ is periodic point of period$\tau$, then we have 

\begin{equation}\label{Lyapunov and Curvature}
\chi^{+}(\theta, \eta)\leq \sqrt{-\frac{1}{\tau}\int_{0}^{\tau}K(\gamma_{\theta}(t))ds}.
\end{equation}

\section{Rigidity in dimension 2.}


In this section, assume that $(M,g)$ is a compact surface with a Riemannian metric $g$.



\noindent In order to prove the Theorem \ref{teo4.6}, we will use two important results. The first result is about the approximation of Lyapunov exponents of an invariant measure by  Lyapunov exponents of measures concentrated on periodic orbits (cf. \cite{Kal}) and the second result is about the rigidity of smooth volume being the MME for three-dimensional Anosov flow (cf. \cite{SL}).


For each periodic point $\theta$, let $\mu_{\theta}$ be the unique $\phi^{t}$-invariant probability measure supported on the orbit of $\theta$, which may be obtained as the normalized push-forward of Lebesgue measure on $\mathbb{R}$ by the map $t\to \phi^{t}(\theta)$.
\begin{teo}\label{teo4.7}\emph{\textbf{[Kalinin]}}
 Let $\phi^t$ be a hyperbolic flow on a $n$-dimensional manifold $N$ and $\mathcal{A}$ a H$\ddot{o}$lder continuous cocycle over $\phi^{t}$. Let $\mu$ be an ergodic $\phi^{t}$-invariant measure and let $\lambda_{1}\leq\lambda_{2}\leq\cdots\leq\lambda_{n}$ be the Lyapunov exponents of $\mathcal{A}$ with respect to $\mu$, counted with the multiplicity. Then for every $\varepsilon>0$, there is a periodic point $\theta\in N$ of $\phi^{t}$ such that the Lyapunov exponents $\lambda_{1}^{\theta}\leq\lambda_{2}^{\theta}\leq\cdots\leq\lambda_{n}^{\theta}$ of $\mathcal{A}$ with respect to $\mu_{\theta}$ satisfy $$|\lambda_{i}-\lambda_{i}^{\theta}|<\varepsilon,$$
for each $1\leq i\leq n$.
\end{teo}

We use Kalinin's result in the special case where $N=SM$, $\phi^t$ is the geodesic flow and the cocycle $\mathcal{A}$ is the cocycle derivative of the geodesic flow $\phi^{t}$, this means:
\begin{eqnarray*}
\mathcal{A}:N \times \mathbb{R}&\to & GL(n,\mathbb{R})\\
(x,t)&\mapsto &\Lambda(x,t)=d_{x}\phi^{t},
\end{eqnarray*}
and $n=2\dim(M)-1$.

\noindent To announce the second result we need, we start with the definition of \textit{Algebraic flows}, which can be found in \cite{SL} and \cite{t}.

\begin{defi}\label{defi4.8}
An Anosov flow $\Phi:N\to N$ on a $3$-dimensional compact manifold $N$ is \emph{algebraic} if it is finitely covered by
\begin{itemize}
\item[\emph{(1)}] a suspension of a hyperbolic automorphism of the $2$-torus $\mathbb{T}^{2}=\mathbb{R}^{2}/\mathbb{Z}^{2}$;
\item[\emph{(2)}] or the geodesic flow on some closed Riemannian surface of constant negative curvature.
\end{itemize}

\end{defi} 

The following result can be found in \cite{SL}, it gives us a characterization of the conjugacy of a flow to an algebraic flow when we have that the volume measure is a measure of maximal entropy.
\begin{teo}\label{teo4.9}\emph{[\textbf{SLVY}]}
Let $k\geq 5$ be some integer and let $\Phi$ be a $C^{k}$ Anosov flow on a compact connected $3$-manifold $N$ such that the smooth volume measure $\mu$ is invariant. Then $h_{top}(\Phi)=h_{\mu}(\Phi)$ if and only if $\Phi$ is $C^{k-\epsilon}$-conjugate to an algebraic flow, for $\epsilon>0$ arbitrarily small.
\end{teo}

The Theorem \ref{teo4.9} is related to the following conjecture:

\begin{conj}\label{conj411}
\emph{[\textbf{Katok Entropy Conjecture}]} Let $(M,g)$ be a connected Riemannian manifold of negative curvature and $\psi$ be the corresponding geodesic flow. Then $h_{top}(\psi)=h_{\mathcal{L}}(\psi)$ if and only if $(M,g)$ is a locally symmetric space, where  $\mathcal{L}$ is the Liouville measure in $SM$.
\end{conj}



Finally, we mention the following result due to Plante (see \cite{Pl}). Basically, the result gives us two alternatives for the strong stable and strong unstable manifold:  they are or not dense on the manifold when the Anosov flow $\psi^t$ is transitive or equivalently the non-wondering set $\Omega(\psi^t )=M$.

\begin{teo}\label{teo4.11}
\emph{[\textbf{Plante}]} Let $\psi^{t}:M\to M$ be an Anosov flow such that $\Omega(\psi^t)=M$. Then there are two possibilities:

\begin{enumerate}
\item[\emph{(a)}] Each strong stable and strong unstable manifold is dense in $M$, or
\item[\emph{(b)}] $\psi^{t}$ is a suspension (modulo time scale change by a constant factor) of an Anosov diffeomorphism of a compact $C^{1}$ submanifold of codimension one in $M$.
\end{enumerate} 
\end{teo}

We use the Theorem \ref{teo4.9}, Theorem \ref{teo4.11}, Ruelle's inequality, and Pesin's formula to prove the Theorem \ref{teo4.6} discarding some cases that may appear for our geodesic flow.

\subsection{Rigidity Theorems}
In this section, we prove the Theorem \ref{teo4.14}, we use it to prove Theorem \ref{New Theorem}, and we use the structure of the geodesic flow to prove Theorem \ref{teo4.6}. Finally, we prove Theorem \ref{Teo_New_1} using techniques of surfaces without focal points. 

\begin{proof}[\emph{\textbf{Proof of Theorem \ref{teo4.14}}}] Denote by $\mathcal{M}_{e}(\phi^t)$ the set of all ergodic $\phi^{t}$-invariant measures.

From Oseledets' ergodic theorem (see \cite{Osel}), for $\mu\in\mathcal{M}_{e}(\phi^t)$ consider $\lambda_{1}\leq \cdots \leq \lambda_{n-1}\leq 0\leq \lambda_{n+1}\leq \cdots\leq \lambda_{2n-1}$ being the Lyapunov exponents associated with $\mu$. Let $\theta\in Per(\phi^t)$ and $\mu_{\theta}\in\mathcal{M}_{e}(\phi^t)$ the unique $\phi^{t}$-invariant probability measure supported on the orbit of $\theta$. Thus, consider $\lambda_{1}^{\theta}\leq \cdots \leq \lambda_{n-1}^{\theta}\leq 0\leq \lambda_{n+1}^{\theta}\leq \cdots\leq \lambda_{2n-1}^{\theta}$ the Lyapunov exponents associated to $\mu_{\theta}$. By hypothesis $\lambda_{i}^{\theta}=-\alpha$, $i\in\{1,\dots, n-1\}$ and $\lambda_{i}^{\theta}=\alpha$, $i\in\{n+1, \dots, 2n-1\}$. From Theorem \ref{teo4.7} we can approximate the Lyapunov exponents of $\mu$ by Lyapunov exponents of $\mu_{\theta}$. Thus, we can conclude that for $\mu\in\mathcal{M}_{e}(\phi^t)$ holds $\lambda_{i}=-\alpha$, $i\in\{1,\dots, n-1\}$ and $\lambda_{i}=\alpha$, $i\in\{n+1, \dots, 2n-1\}$.\\
Since $\phi^t$ is Anosov, then normalize Liouville measure $\mathcal{L}$ is an element of $\mathcal{M}_{e}(\phi^t)$ (see \cite{VO}). 

Now we will show that the Liouville measure on $SM$ is a maximal measure of entropy. Indeed, by Ruelle's inequality, for $\mu\in\mathcal{M}_{e}(\phi^t)$ we have that 
\begin{eqnarray}\label{ec412}\nonumber
h_{\mu}(\phi)&\leq &\int_{SM}\sum_{\lambda_i>0} \lambda_i\, d\mu(\theta)\\\nonumber
&=&\int_{SM}(n-1)\alpha\, d\mu(\theta)\\
&=&(n-1)\alpha.
\end{eqnarray}

\noindent Since $\phi^{t}$ is $C^{2}$, by Pesin's formula we obtain that
\begin{eqnarray}\label{ec413}\nonumber
h_{\mathcal{L}}(\phi)&=&\int_{SM}\sum_{\lambda_i>0} \lambda_i\,(\theta)d\mathcal{L}(\theta)\\\nonumber
&=&\int_{SM}(n-1)\alpha d\mathcal{L}(\theta)\\
&=&(n-1)\alpha.
\end{eqnarray}

\noindent Thus, from (\ref{ec412}) and (\ref{ec413}), we conclude that for all $\mu\in\mathcal{M}_{e}(\phi^t)$ holds $$h_{\mu}(\phi)\leq (n-1)\alpha=h_{\mathcal{L}}(\phi).$$
Thus, from the variational principle 
\begin{equation*}\label{ec4.11}
h_{top}(\phi)=\sup_{\mu\in\mathcal{M}_{e}(\phi^t)}h_{\mu}(\phi)\leq (n-1)\alpha=h_{\mathcal{L}}(\phi)\leq h_{top}(\phi),
\end{equation*}
which concludes the proof of the theorem.
\end{proof}

The proof of Theorem \ref{teo4.14} implies the following corollary.
\begin{coro}\label{CorNew2} If $\phi^t\colon N \to N$ is a tridimensional conservative Anosov flow, then if the Lyapunov exponents are constant along periodic orbits, then the volume measure is a measure of maximal entropy.
\end{coro}
\begin{proof}
The proof follows the same arguments of the proof of Theorem \ref{teo4.14} and the ergodicity of the volume measure (see \cite{VO}).
\end{proof}

\begin{proof}[\emph{\textbf{Proof of Theorem \ref{New Theorem}}}] \ \\
$(1)\implies(2)$ is a consequence of Corollary \ref{CorNew2}.
$(2)\implies(3)$ is a consequence of Theorem \ref{teo4.9}.
In order to prove that $(3)\implies (1)$, observe that the condition $(1)$ is preserved by smooth conjugacy. Finally, we note that any algebraic tridimensional flow satisfies the condition $(1)$, therefore the condition $(1)$ is satisfied. 
\end{proof}
To the proof of Theorem \ref{teo4.6} we need to use the classification of transitive tridimensional Anosov flow,  the structure of the geodesic flow. 

\begin{proof}[\emph{\textbf{Proof of Theorem \ref{teo4.6}}}]
It is easy to see that $(2)\implies(1)$. From Theorem \ref{New Theorem} our flow is smooth conjugate to an algebraic model. However, the geodesic flow has no global cross-section (see \cite{Pa}), therefore from Theorem \ref{teo4.11}, our flow must be smooth conjugate to the geodesic flow of a surface of constant negative curvature. Finally, from  \cite{Cr} our surface must be isometric to a surface of constant negative curvature, in other words, it has constant negative curvature, and the value of the Lyapunov exponents determines the value of the curvature.
\end{proof}
The remainder of this section will be dedicated to the proof of Theorem \ref{Teo_New_1}.

\ \\
We denote by $$\Gamma(\kappa):=\{\theta\in SM: K(\pi(\theta))\leq -\kappa^2\}$$
The next lemma is similar to  \cite[Lemma 2.4]{KB} and \cite[Lemma 2.6]{PW}.
\begin{lem}[Shadowing Lemma]\label{Shad_Lem}
Let $M$ be a closed surface without focal points. For any $\eta, \epsilon, \tau>0$ there exists $\delta>0$ such that for any collection of orbit segments
$\{(\theta_i,t_i)\}_{i\in \mathbb{Z}}$ with $\theta_i, \phi^{t_i}(\theta_i)\in \Gamma(\kappa)$, $t_i\geq \tau$ and  $d(\phi^{t_i}(\theta_i), \theta_{i+1})<\delta$ for all $i\in \mathbb{Z}$, there exist a geodesic $\gamma$  and a sequence of times $\{T_i\}_{i\in  \mathbb{Z}}$ with $T_0=0$, $T_i+t_i-\epsilon \leq T_{i+1} \leq T_i + t_i +\epsilon$, and $d(\gamma'(t), \gamma'_{\theta_i}(t-T_i))<\epsilon$ for all $t\in [T_i,T_{i+1}]$ and $i\in  \mathbb{Z}$.\\
The geodesic $\gamma$ is unique up to re-parametrization. Moreover, if the orbits being shadowed are periodic, then the shadowing orbit is also periodic.
\end{lem}

\begin{defi} Let $M$ be a complete manifold without focal points. For each $\theta\in SM$, we define $\text{rank}\,(\theta)$ as the dimension of the vector space of the parallel Jacobi fields along the geodesic $\gamma_{\theta}$, and $\text{rank}(M):=\min\{\text{rank}(\theta): \theta\in SM\}$. 
\end{defi}

It is not difficult to see that for manifold without focal points $\text{rank}(\theta)$ is the codimension of the vector bundle $G^s_{\theta}\oplus G^u_{\theta}$. Consider the following set 
$$\mathcal{H}=\{\theta\in SM: K(\gamma_{\theta}(t))=0, \,\, \text{for all}\,\, t\in \mathbb{R}\}.$$

The set $\mathcal{H}$ is clearly invariant to the geodesic flow. In general, the set $\mathcal{R}=SM\setminus \mathcal{H}$ is called \emph{regular set}. It is well known that, in the case of a surface without focal points, the Liouville measure is ergodic for $\phi^t$ in $\mathcal{R}$ (see \cite{Barr-Pes}). However, there is an open conjecture about the ergodicity of the Liouville measure over the whole surface. Such a conjecture is equivalent to the set $\mathcal{R}$ having total Liouville measure or the set $\mathcal{H}$ having zero Liouville measure (see \cite{Burns}).

It is easy to see that for surfaces without focal points $$\mathcal{H}=\{\theta\in SM: G^s_{\theta}\cap G^u_{\theta}\neq \{0\}\}.$$
Therefore  $\mathcal{H}=\emptyset$ is equivalent to geodesic flow $\phi^t$ to be Anosov (see \cite{Ebe}).
 Observe that in the case of surface, $\text{rank}(\theta)$ only can be $1$ or $2$. It is easy to see that if $\theta\in \mathcal{H}$, then $\text{rank}(\theta)=2$ and if $\theta\in SM\setminus \mathcal{H}$,  then $\text{rank}(\theta)=1$.

\begin{lem}\label{transtivity}Let $M$ be a compact surface of genus greater than $1$ and without focal points, then the geodesic flow $\phi^t: SM \to SM$ is transitive.
\end{lem}
\begin{proof}
We only need to prove that $M$ has rank $1$ and then the result will become a consequence of \cite{LZ}. As $\text{genus of} \, M >1$, then the Euler characteristic of $M$ is negative, then the topological entropy $h_{top}(\phi^t)>0$ (cf. \cite{Di}).  Consequently, from variational principle,  there is an ergodic measure with positive Lyapunov exponents, \emph{i.e} there is $\theta\in SM$ and $\eta\in T_\theta SM$ such that $\chi^{+}(\theta, \eta)>0$. Using the above notation, 
$$\chi^{+}(\theta, \eta)=\lim_{t\to \infty}\frac{1}{t}\log ||J_{\eta}(t)||>0.$$
Therefore,  $\text{rank}(\theta)=1$. In fact, $\theta\in SM \setminus \mathcal{H}$ and 
$\text{rank}(M)=1$.  
\end{proof}

\begin{remark}The proof of the last lemma can be easier using the Gauss-Bonnet theorem since it should have points in $M$ of negative curvature. Therefore, together with the condition of no focal points, it implies the existence of an unbounded unstable Jacobi field. In particular, we have a point with rank $1$. Also, using arguments of Eberlein \emph{(see \cite{Ebe})}, we can see easily that parallel Jacobi field along a geodesic $\gamma$ implies section curvature zero along $\gamma$, so geodesic pass through point of negative curvature provides points of rank $1$.
\end{remark}

The next lemma is similar to at \cite[Proposition 3]{KB}. and \cite[Proposition 3.4]{PW}.

\begin{lem}\label{Lem_New_1} If $ \mathcal{H}\neq \emptyset$, then there is a periodic orbit with Lyapunov exponent arbitrarily close to $0$.
\end{lem}
\begin{proof}
By hypotheses there is $\theta\in  \mathcal{H}$ and consequently $K(\gamma_{\theta}(t))=0$, for all $t\in \mathbb{R}$. Given $\kappa>0$, by transitivity (see Lemma \ref{transtivity}), there is a sequence $\theta_k$ converges to $\theta$ and unbounded negative and positive sequences $t_k^{-}$ and $t_k^{+}$, respectively,  such that 
$$K(\gamma_{\theta_{k}}(t))\leq -\kappa^2, \,\,\,\text{for all} \,\,\, t\in[t_k^-, t_k^+].$$ 
By compactness, we can assume that $\phi^{t_{k}^{-}}(\theta_k)$ and $\phi^{t_{k}^{+}}(\theta_k)$ converge to $w^{-}$ and $w^{+}$, respectively. 
Note that $K(\pi(w^-))\leq -\kappa^2$ and $K(\pi(w^{+}))\leq -\kappa^2$. \\
Consider $\epsilon$ and $\delta$ form Lemma \ref{Shad_Lem} (Shadowing Lemma), then, using the transitivity of the flow, there is $T>\tau$ and $w$ with $d(w,w^+)<\frac{\delta}{2}$ such that $d(\phi^{T}(w),w^-)<\frac{\delta}{2}$.
Thus, from Lemma \ref{Shad_Lem}, for any large enough $k$ there is a closed geodesic $\beta_k$ of period $P_k$, close to $t_k^{+}-t_{k}^{-}+T$ such that $(\tau,\epsilon)$-shadows the pseudo-orbit formed by $\phi^t(\theta_{k})|_{[t_k^-, t_k^+]}$ and $\phi^{t}(w)|_{[0,T]}$.
It is easy to see that $P_k$ is an unbounded sequence, moreover the path of orbit $\phi^{t}(w)|_{[0,T]}$ such that $\beta_{k}$ $(\tau,\epsilon)$-shadows does not depend on $k$, then the proportion of the curvature along $\beta_k$ is bounded. Moreover, by construction, the curvature of $\beta_k$ is close to the curvature of $\pi(\phi^t(\theta_{k})|_{[t_k^-, t_k^+]})$ most part of time. In other words, 
$$\limsup_{k\to \infty}\frac{1}{P_k}\int_{0}^{P_k}K(\beta_{k}(t))dt\geq -2\kappa^2.$$
Put $\beta_{k}(0)=w_k$ and $\eta\in G^{u}_{w_k}$, then from (\ref{Lyapunov and Curvature})

$$0\leq \chi^{+}(w_k, \eta)\leq \sqrt{-\frac{1}{P_k}\int_{0}^{P_k}K(\beta_{k}(t))dt}\leq \sqrt{2}\kappa.$$
Since $\kappa$ was chosen arbitrarily,  we have our result.

\end{proof}


\begin{proof}[\emph{\textbf{Proof of Theorem \ref{Teo_New_1}}}]
It is easy to see that $(2)\implies(1)$. We prove $(1)\implies(2)$ using the genus of the surface $M$. Note that $M$ has no focal points, in particular, has no conjugate points. Thus, from Hopf's result (cf. \cite{hopf}), $M$ has genus $\geq 1$. In the case of genus equal to $1$, from \cite{hopf}  our surface $M$ is the flat torus $\mathbb{T}^2$ and consequently $\alpha=0$. So, we can assume that $M$ has a genus greater than $1$. \\
\textbf{Claim:} The set $ \mathcal{H}= \emptyset$.
In fact, by contradiction, assume that there is $\theta \in \mathcal{H}$. Therefore, from Lemma \ref{Lem_New_1}, we have periodic orbits with Lyapunov exponent arbitrarily close to $0$, however since the value of $\alpha$ in the hypothesis does not depend on the periodic points, then we have a contradiction. \\
From the last Claim, all geodesics pass through negative curvature. Thus, as $M$ has no focal points then $\phi^t$ is Anosov (cf. \cite{Ebe}). Finally, since the Lyapunov exponents are constant along periodic orbits, from Theorem \ref{teo4.14} the Liouville measure is a measure of maximal entropy. Moreover, as $M$ is $C^2$ Riemannian metric without focal points, then \cite[Corollary 3.3]{Katok82} we conclude that $M$ has constant negative curvature and the value of the curvature is $-\alpha^2$, as we wished. 
\end{proof}

\begin{remark}In the final arguments of the proof of Theorem \ref{Teo_New_1}, even the geodesic flow becomes Anosov, we cannot use \emph{Theorem \ref{teo4.6}} since in the assumptions of \emph{Theorem \ref{Teo_New_1}} our metric is just $C^2$, different from \emph{Theorem \ref{teo4.6}} where the metric is at least $C^5$.
\end{remark}

On the hypothesis of Theorem \ref{teo4.6} over the geodesic flow be $C^{k}$, $k\geq 5$, we believe the theorem is true for $k\geq 2$, because observation of the Theorem \ref{teo4.9} on \cite{SL} says that this theorem can be true for $k\geq 2$, but technical obstructions prevent finding the precise boundary of required regularity. 

\section{Rigidity of Lyapunov exponents in finite volume case.}
The main goal of this section is to prove the Theorem \ref{thm4.5}. Therefore, from now on, we  consider $M$ a complete Riemannian manifold with pinched negative curvature, \emph{i.e.}, there are two constants $0<b<c$ such that $-c^{2}\leq K_{M}\leq -b^{2}$. For  $\theta\in SM$, let  $Y_{\theta,s}(t)$ and $Y_{\theta,u}(t)$ the stable and unstable Jacobi tensor along $\gamma_{\theta}(t)$, respectively and $U_{\theta,s}(t)=Y'_{\theta,s}(t)Y^{-1}_{\theta,s}(t)$ and $U_{\theta,u}(t)=Y'_{\theta,u}(t)Y^{-1}_{\theta,u}(t)$ the stable and unstable Riccati solution, respectively.

\subsection{Proof of the Theorem \ref{thm4.5}}
One of the main ingredients of the proof of Theorem \ref{thm4.5} is the following theorem:

\begin{teo}[\cite{Nina}]\label{teo2.7}
If $\phi^t\colon SM\to SM$ is an Anosov geodesic flow and $M$ has finite volume, then $\overline{Per(\phi^t)}=SM$.
\end{teo}

\begin{proof}[\emph{\textbf{Proof of Theorem \ref{thm4.5}}}]
\ \\
The proof of Theorem \ref{thm4.5} carried out in two cases:

\noindent \textbf{Case 1:} We consider the case $\alpha=b$.\\
Note that the hypothesis $\chi^{+}(\theta,\xi)=b$ for all $\theta\in \text{Per}(\phi^t)$ and all $\xi\in E_{\theta}^{u}$.  So, $$\lim_{t\to+\infty}\dfrac{1}{t}\log|\det d_{\theta}\phi^{t}\vert_{ E_{\theta}^{u}}|=b\cdot \dim E_{\theta}^{u}=b(n-1).$$

\noindent From \cite[Lemma 3.2]{IR}, we obtain 
\begin{eqnarray}\label{ec4.10}
\lim_{t\to+\infty}\dfrac{1}{t}\int_{0}^{t}\text{tr}(U_{\theta, u}(s))ds=\lim_{t\to+\infty}\dfrac{1}{t}\log|\det d_{\theta}\phi^{t}|_{E_{\theta}^{u}}|=b(n-1),
\end{eqnarray}
for all $\theta\in \text{Per}(\phi^t)$.\\
However, if $\theta\in \text{Per}(\phi^t)$ of period $\tau$, then the function $t\to \text{tr}(U_{\theta, u}(t))$ is periodic of period $\tau$. From (\ref{ec4.10}) we have 
\begin{equation}\label{ENEW1}
\dfrac{1}{\tau}\int_{0}^{\tau}\text{tr}(U_{\theta, u}(s))ds=b(n-1).
\end{equation}

On the other hand, as $K\leq -b^{2}$ and $U_{\theta,u}(t)$ is symmetric, from (\ref{ec4.8}) we can conclude that all eigenvalues of $U_{\theta,u}(t))$ are non-negative, in fact, all of them are greater than or equal to $b$. Let $\lambda_{n-1}(t)\geq\lambda_{n-2}(t)\geq\cdots\geq\lambda_{1}(t)\geq b$ the eigenvalues of $U_{\theta,u}(t)$. Then
\begin{equation}\label{ENEW2}
 \text{tr}(U_{\theta,u}(t))=\sum_{i=1}^{n-1}\lambda_{i}(t)\geq \alpha(n-1).
\end{equation}

\noindent From (\ref{ENEW1}) and (\ref{ENEW2})  $$\text{tr}(U_{\theta,u}(t))=b(n-1)$$
for all $t\in\mathbb{R}$.\\
Hence, for all $j=1,2,\dots,n-1$ holds $\lambda_{j}(t)=b$ for all $t\in\mathbb{R}$. Finally this implies for all $t\in\mathbb{R}$ 
\begin{equation}\label{eq1New*}
\text{tr}(U_{\theta,u}(t))^{2})=\sum_{j=1}^{n-1}(\lambda_{j}(t))^{2}=b^{2}(n-1).
\end{equation}
From the Riccati equation 
$$\frac{1}{t}\int_{0}^{t}\text{tr}(U'_{\theta,u}(s))ds+\dfrac{1}{t}\int_{0}^{t}\text{tr}(U_{\theta,u}(s)^{2})ds + \dfrac{(n-1)}{t}\int_{0}^{t}\text{Ric}(\phi^{s}(\theta))=0$$
Taking $t$ goes to $\infty$, then (\ref{EQ-Green}) and (\ref{eq1New*}) provide 
\begin{equation*}
b^{2}(n-1)=-\lim_{t\to+\infty}\dfrac{(n-1)}{t}\int_{0}^{t}\text{Ric}(\phi^{s}(\theta))ds
\end{equation*}
for all $\theta\in \text{Per}(\phi^t)$.
If $\theta\in Per(\phi^t)$ of period $\tau$, then  the last equation becomes 
$$\dfrac{1}{\tau}\int_{0}^{\tau}\text{Ric}(\phi^{s}(\theta))ds=-b^{2}.$$

Note that  $K\leq -b^{2}$, in particular, $\text{Ric}(\phi^s(\theta))\leq -b^2$. Consequently, the last equation provides $\text{Ric}(\phi^{t}(\theta))=-b^{2}$, for all $t\in\mathbb{R}$. In resume, we had proved that for all $\theta\in \text{Per}(\phi^t)$ holds $\text{Ric}(\theta)=-\alpha^{2}$. As $M$ has finite volume, then from Theorem \ref{teo2.7} the periodic orbits are dense in $SM$, so the continuity of the function $\text{Ric}(\cdot)$ guarantees   $\text{Ric}(\theta)=-b^{2}$, for all $\theta\in SM$ and \emph{a fortiori},  $K_{M}=-b^{2}$.\\
\ \\
\noindent \textbf{Case 2:}  If $\alpha=c$. Here we use the same arguments of \cite{IR}. By the hypothesis $\chi^{+}(\theta,\xi)=c$ for all $\theta\in \text{Per}(\phi^t)$ and all $\xi\in E_{\theta}^{u}$. Thus 

$$\lim_{t\to+\infty}\dfrac{1}{t}\log|\det d_{\theta}\phi^{t}\vert_{ E_{\theta}^{u}}|=c\cdot\text{ dim} E_{\theta}^{u}=c(n-1).$$

\noindent If $\theta\in \text{Per}(\phi^t)$ of period $\tau$, then from \cite[Lemma 3.2]{IR}, we obtain 
\begin{eqnarray}\label{ec4.21}
\frac{1}{\tau}\int_{0}^{\tau}\text{tr}(U_{\theta,u}(s))ds=\lim_{t\to+\infty}\dfrac{1}{t}\log\left|\det d_{\theta}\phi^{t}|_{E_{\theta}^{u}}\right|=c(n-1).
\end{eqnarray}
Since $U^{u}(\phi^{s}(\theta))$ is symmetric then it is easy to see that
\begin{equation}\label{Eq2New*}
(\text{tr}(U_{\theta,u}(s))^{2}\leq (n-1)\text{tr}(U_{\theta,u}(s)^{2}).
\end{equation} 
Using the periodicity of the function $s \to tr(U_{\theta,u}(s))$ and  integrating from $0$ to $\tau$ the Riccati equation (\ref{ec1.3}), we obtain from (\ref{ec4.21}), (\ref{Eq2New*}), and Cauchy-Schwartz inequality that
\begin{eqnarray*}
c(n-1)=\dfrac{1}{\tau}\int_{0}^{\tau}\text{tr}(U_{\theta,u}(s))ds &\leq & \sqrt{\dfrac{1}{\tau}\int_{0}^{\tau}(\text{tr}(U_{\theta,u}(s))^{2}ds}\\
&\leq & \sqrt{\dfrac{(n-1)}{\tau}\int_{0}^{\tau}\text{tr}(U_{\theta,u}(s)^{2})ds}\\
&=&\sqrt{-\dfrac{(n-1)}{\tau}\int_{0}^{\tau}\text{tr}(R(s))ds}\\
&=&\sqrt{-\dfrac{(n-1)^{2}}{\tau}\int_{0}^{\tau}\text{Ric}(\phi^{s}(\theta))ds}\\
&\leq &\sqrt{(n-1)^{2}c^{2}},
\end{eqnarray*}
where in the last inequality we used that the sectional curvature satisfies $K_{M}\geq -c^{2}$.
Consequently, $$\dfrac{1}{\tau}\int_{0}^{\tau}\text{Ric}(\phi^{s}(\theta))ds=-c^2,$$
which implies that $\text{Ric}(\phi^{t}(\theta))=-c^{2}$ for all $\theta \in \text{Per}(\phi^t)$ and all $t\in \mathbb{R}$. The result follows the same lines of the case 1.




\end{proof} 

\textbf{Acknowledgements:} The first author thanks all the members of the Institute of Mathematics - UFRJ for all their support during the preparation of this work, and my family for being by my side whenever I needed them, I am so sorry for leaving so early. The second author thanks the first author for all his struggles in recent years, his memory will always be with me. The second author also thanks the Mathematics Department of SUSTech - China for its hospitality during the preparation of this article.

\bibliographystyle{plain}
\pagestyle{empty}

\end{document}